\documentclass[reqno,11pt]{amsart}
\usepackage[colorlinks=true,linkcolor=blue,citecolor=blue]{hyperref}
\usepackage{graphicx}
\newtheorem{definition}{Definition}[section]
\newtheorem{rem}[definition]{Remark}
\newtheorem{prop}[definition]{Proposition}
\newtheorem{lem}[definition]{Lemma}
\newtheorem{coro}[definition]{Corollary}
\newtheorem{teo}[definition]{Theorem}

\numberwithin{equation}{section}

\def\R{{\mathbb{R}}}

\def\epsilon{{\varepsilon}}
\def\phi{{\varphi}}
\def\theta{{\vartheta}}

\begin{document}

%\begin{document}

\title{{\bf Lines of curvature of the double torus}}
\author[M. Garc\'{\i}a Monera, V. G\'omez Guti\'errez and 
F. S\'anchez-Bringas]{Mar\'{\i}a Garc\'{\i}a Monera, Vinicio G\'omez Guti\'errez and \\
Federico S\'anchez-Bringas %\\
%Facultad de Ciencias, UNAM,
%Ciudad Universitaria \\
%04510 M\'exico D.F., M\'exico
}
\address{M. Garc\'{\i}a Monera, Departamento de Did\'actica de la Matem\'atica, Universitat de Valencia, 46100 Burjassot, Valencia, Espa\~na.}
\email{monera2@uv.es}
\address{Vinicio G\'omez Guti\'errez, Facultad de Ciencias, UNAM, Ciudad Universitaria, c.p. 04510 M\'exico D.F., M\'exico.}
\email{vgomez@ciencias.unam.mx}
\address{Federico S\'anchez-Bringas, Facultad de Ciencias, UNAM, Ciudad Universitaria, c.p. 04510 M\'exico D.F., M\'exico.}
\email{sanchez@unam.mx}

\maketitle

\begin{abstract}
We describe the $\nu$-lines of curvature of an embedding of the double torus
into $\mathbb R^4$, defined as the link of the real part of the Milnor fibration of a
polynomial, where $\nu$ is its gradient. Through this analysis, we present a
complete description of the foliation of lines of curvature of the
embedding, defined as the image of the stereographic projection of this
link into $\mathbb R^3$.

\bigskip

\end{abstract}

\vspace{1cm}

%%%%%%%%%%%%%%%%%%%%%%%%%%%%%%%%%%%%%%%%%%%%%%%%%%%%%%%%%%%%%%%%%%%%%%%%%%%%%%%%%%%%%%%%%%%%%%%%%%%%%%%%%%%%%%%%%%%%%%%%%%%%%%%%%%%%%%%%%%%%%%%%%%%%%%%%%%%%%%%

\section{Introduction}

%%%%%%%%%%%%%%%%%%%%%%%%%%%%%%%%%%%%%%%%%%%%%%%%%%%%%%%%%%%%%%%%%%%%%%%%%%%%%%%%%%%%%%%%%%%%%%%%%%%%%%%%%%%%%%%%%%%%%%%%%%%%%%%%%%%%%%%%%%%%%%%%%%%%%%%%%%%%%%%

The lines of curvature of surfaces in $\mathbb R^3$ have been a subject of interest for several centuries, since G. Monge described 
the case of the triaxial ellipsoid by direct integration     
of the differential equation of these lines. This is the first example of a foliation with singularities  
on a compact surface \cite{Monge}. 
Despite the fact that many studies have focused on different properties of local and global nature of the lines of curvature of surfaces in $\mathbb R^3$, 
(see for instance \cite{Soto}, \cite{Bruce}, \cite{S-X}, \cite{Hopf}, \cite{Spi}),
a complete analytic description of them has not been achieved for the case of an oriented compact surface of genus higher than 
one. A reason for this may be that, contrasting with the ellipsoid case which is a quadric surface,   
embeddings of compact surfaces of higher genus are quartic or even higher degree algebraic surfaces, \cite{Griffits, Hirsch} and consequently, their  
lines of curvature are described by rather complicated differential equations.
%\footnote{see \cite{Griffits} page 12 and \cite{Hirsch} page 28}
The simplest case in this class of surfaces is the double torus. 

In the present article we present a description of the foliation of lines of curvature 
of an embbeding of the double torus into $\mathbb R^3$.   
In order to avoid the problem arising from the degree of the algebraic surface in $\mathbb R^3$
we consider an embedding of the double torus into $\mathbb R^4$ as the link of the real part of the Milnor fibration of a polynomial with isolated singularity at the origin.
Specifically, this surface is embedded into the sphere $\mathbb S_r^3 \subset \mathbb R^4$ of radius $r$ with center at the origin for which
we analyze the lines of curvature with respect to the gradient $\nu$, referred to in our research as $\nu$-lines of curvature, of this polynomial
on the link.   
   
The $\nu$-lines of curvature of surfaces, where $\nu$ is a normal vector field,
have been studied in $\mathbb R^4$ in \cite{F} and  \cite{Ga-S}, where the expression of the differential equation is given in the Monge chart, in which 
the surface is defined as the graph of a differentiable function. 
In this article, we introduce an equation of these lines according to the description of the surface as the intersection of the loci of two polynomials,
which simplifies the analysis. This equation, together with the symmetries of the link, allow us to apply elements from both algebraic geometry and dynamical systems
to describe the foliation completely. Then, we apply the stereographic projection which transforms the $\nu$-lines of curvature of this spherical embedding into the lines of curvature of the image of the link in $\mathbb R^3$.

The article is organized as follows: in section 2 we present some preliminaries and the formulation of the 
differential equation of $\nu$-lines of curvature that will be used in the study (Proposition \ref{3.2} and Corollary \ref{equationframe}). In section 3, we analyze the embbeding of the double torus into $\mathbb R^4$. Namely, we determine a decomposition of the surface
parameterized with convenient coordinate charts which allow us to use the symmetries of the surface to simplify the study. In section 4, we prove the main theorem 
which provides the description of the foliation of the $\nu$-lines of curvature on the double torus embbeded into $\mathbb R^4$ as a transversal intersection (Theorem \ref{Maintheorem}). 
As a corollary we get the description of the lines of curvature of the image of this embedding under the stereographic projection in $\mathbb R^3$ 
(Corollary \ref{torusinR3}).

%\footnote{The double torus and moreover, the multiple tori considered in the present paper immersed in the 3-dimensional sphere, are not minimal surfaces
%(in $\mathbb R^4$...)}
%We know that for these tori $K_N=0$ and $\Delta<0$. What can we say about $H$ and $K$?. 
%Indeed, the vanishing of the mean curvature vector together with the condition of semiumbilicity satisfied everywhere in the surfaces imply that all the points  
%are inflection points. For this type of surfaces the principal configuration with respect to any normal direction is trivially null. 
%We will see in the sequel that these tori have different type of principal configurations. The geometry of tori considered as minimal surfaces in the  
%3-dimensional sphere is studied by several authors, a landmark article on the subject is \cite{Lawson}.}

%%%%%%%%%%%%%%%%%%%%%%%%%%%%%%%%%%%%%%%%%%%%%%%%%%%%%%%%%%%%%%%%%%%%%%%%%%%%%%%%%%%%%%%%%%%%%%%%%%%%%%%%%%%%%%%%%%%%%%%%%%%%%%%%%%%%%%%%%%%%%%%%%%%%%%%%%%%%%%%

\section{Preliminaries}

%%%%%%%%%%%%%%%%%%%%%%%%%%%%%%%%%%%%%%%%%%%%%%%%%%%%%%%%%%%%%%%%%%%%%%%%%%%%%%%%%%%%%%%%%%%%%%%%%%%%%%%%%%%%%%%%%%%%%%%%%%%%%%%%%%%%%%%%%%%%%%%%%%%%%%%%%%%%%%%
%\textcolor{blue}{ejemplo para colorear}

\subsection{The $\nu$-lines of curvature of a surface in $\R^4$.}

Let $M$ be a smooth oriented surface immersed in $\R^4$ with the Riemannian
metric
induced by the standard Riemannian metric of $\R^4$. For each $p \in M$
consider
the decomposition $T_p \R^4 =T_pM \oplus N_pM$, where $N_pM$ is the
orthogonal complement
of $T_pM$ in $\R^4$.
Let $\bar \nabla$ be the Riemannian connection of $\R^4$. Given local
vector fields
$X, \ Y$ on $M$, let $\bar X, \ \bar Y$ be local extensions to $\R^4$.
The tangent component of the Riemannian connection in $\R^4$ is the Riemannian
connection of $M: \nabla_XY=(\bar \nabla_{\bar X}\bar Y)^{\top}.$

Let ${\mathcal X}(M)$ and ${\mathcal N}(M)$ be the space of the smooth
vector fields tangent to $M$ and
the space of the smooth vector fields normal to $M$, respectively.
Consider the {\it second fundamental form},
$$II: {\mathcal X}(M)\times {\mathcal X}(M) \rightarrow {\mathcal N} M,
\ II(X,Y) = \bar \nabla_{\bar X}\bar Y - \nabla_XY.$$
This map is symmetric and bilinear. So, at each point $p$, $II(X,Y)$ only depends on 
$X(p)$ and $Y(p)$.

%Let $p\in M$ and $\nu \in N_pM,\ \nu \neq 0$, define the function$$
%H_{\nu}: T_pM \times T_pM \rightarrow \R,\ H_{\nu}(X,Y)=<\alpha(X,Y),\nu>.$$
%\noindent Then this function is as well symmetric and bilinear. 
Suppose the $\nu$ is a unitary vector field in ${\mathcal N}(M)$. {\it The $\nu$-second fundamental form} of $M$ at $p$ is the quadratic form,
$$II_{\nu}: T_pM \rightarrow \R,\ II_{\nu}(X)= \langle II_p(X,X), \nu \rangle.$$

Recall the {\it $\nu$-shape operator}
$$S_{\nu}: T_pM \rightarrow T_pM,
\ S_{\nu}(X)=-(\bar \nabla_{\bar X}\bar \nu)^{\top} ,$$
\noindent where $\bar \nu$ is a local extension to $\R^4$ of the normal
vector field $\nu$
at $p$ and $\top$ means the tangent component. This operator is 
self-adjoint and, for
any $X,\ Y \in T_pM$, satisfies the following equation:
$$<S_{\nu}(X),Y>=<II(X,Y), \nu>.$$

%\noindent So, the $\nu$-second fundamental form is expressed as $II_{\nu}(X) = <S_{\nu}(X),X>$. 
Thus, for each $p \in M$, there exists an
orthonormal
basis of eigenvectors of $S_{\nu} \in T_pM$, for which the restriction of
the second
fundamental form to the unitary vectors, $II_{\nu}|_{S^1}$, takes its
maximal and
minimal values. The corresponding eigenvalues $k_1,\ k_2$ are the {\it maximal}
and {\it minimal $\nu$-principal curvatures}, respectively. The point $p$ is a
$\nu$-umbilic if the $\nu$-principal curvatures coincide.
Let ${\mathcal U}_{\nu}$ be the set of $\nu$-umbilics in $M$.
For any $p \in M \backslash {\mathcal U_{\nu}}$  there are two $\nu$-principal
directions
defined by the eigenvectors of $S_{\nu}$. These fields of directions are
smooth and
integrable, and therefore they define, in  $M \backslash {\mathcal U_{\nu}}$, two families of orthogonal curves, its integrals,
which are called the  {\it $\nu$-principal lines of curvature}, one
maximal and the other
one minimal. The $\nu$-umbilics are considered as the singularities of these foliations.
%The two orthogonal foliations with the $\nu$-umbilics as its
%singularities form the {\it $\nu$-principal configuration} of $M$.
The differential equation of $\nu$-lines of curvature is
\begin{eqnarray*}
S_{\nu}(X(p))=\lambda(p)X(p)
\end{eqnarray*}

%%%%%%%%%%%%%%%%%%%%%%%%%%%%%%%%%%%%%%%%%%%%%%%%%%%%%%%%%%%%%%%%%%%%%%%%%%%%%%%%%%%%%%%%%%%%%%%%%%%%%%%%%%%%%%%%%%%%%%%%%%%%%%%%%%%%%%%%%%%%%%%%%%%%%%%%%%%%%%%

\subsection{Equation of $\nu$-lines of curvature of surfaces defined implicitly in 4-space}

%%%%%%%%%%%%%%%%%%%%%%%%%%%%%%%%%%%%%%%%%%%%%%%%%%%%%%%%%%%%%%%%%%%%%%%%%%%%%%%%%%%%%%%%%%%%%%%%%%%%%%%%%%%%%%%%%%%%%%%%%%%%%%%%%%%%%%%%%%%%%%%%%%%%%%%%%%%%%%%
We first determine a characterization of a $\nu$-principal direction at a point of a surface 
provided with a normal frame $\{\nu, \mu \}$.

\begin{prop}
Let $M$ be a surface immersed in $\mathbb R^4$.
Assume that $\{\nu, \mu \}$ is a frame of the normal
bundle on $M$.
Consider the unitary vector field $\tilde \nu= \frac{\nu}{||\nu||}$.
If $p \in M$, the vector $X \in T_pM$ is tangent to a $\tilde \nu$- principal direction, if and only if
the following equation holds at $p$:

\begin{eqnarray}\label{equation1}
\langle \bar \nabla_X \nu\wedge X \wedge \mu,\ \nu \rangle=0.
\end{eqnarray}

\end{prop}

\noindent{\bf Proof.} We decompose the covariant derivative as the sum of the following
vector fields:
$$\bar \nabla_X \nu=(\bar \nabla_X\nu)^{\top}+p_{\mu}(X)\mu + p_{\nu}(X)\nu,$$
\noindent where  $p_{\mu},\ p_{\nu}: T_pM \rightarrow \R$ are the projections of
$\bar \nabla_X \nu$ on the lines determined by the normal frame $\{\nu, \mu\}$ at $p$.

Therefore,
\begin{eqnarray}
\langle \bar \nabla_X \nu\wedge X \wedge \mu, \nu \rangle & = & \langle (\bar \nabla_{X} \nu)^{\top} +
p_{\mu}(X)\mu + p_{\nu}(X)\nu\wedge X \wedge \mu, \nu \rangle \nonumber \\
& = & \langle (\bar \nabla_{X} \nu)^{\top} \wedge X \wedge \mu, \nu \rangle.
\end{eqnarray}

Observe that the last expression vanishes if and only if $(\bar \nabla_{X} \nu)^{\top}$ and $X$ are
linearly dependent.

If $\nu$ is unitary, $-(\bar \nabla_{X} \nu)^{\top}=S_{\nu}(X)$ is the $\nu$-shape operator. Then, equation $(\ref{equation1})$ holds
if and only if $S_{\nu}(X)$ and $X$ are linearly dependent; this is equivalent to the fact that $X$ must be an eigenvector of the $\nu$-shape operator.

On the other hand, if $\nu$ is not unitary, then,

\begin{eqnarray*}
(\bar \nabla_{X} \tilde \nu)^{\top} & = & \left(\bar \nabla_{X} \frac{\nu}{||\nu||} \right)^{\top} \\
& = & \left (X \left(\frac{1}{||\nu||}\right)\nu + \frac{1}{||\nu||}(\bar \nabla_{X}\nu) \right)^{\top} \\
& = & \frac{1}{||\nu||}\left(\bar \nabla_{X}\nu \right)^{\top}.
\end{eqnarray*}

\noindent Therefore, from this equation we have
\begin{eqnarray*}
\langle \bar \nabla_X \nu\wedge X \wedge \mu, \nu \rangle & = &
\langle (\bar \nabla_{X} \nu)^{\top} \wedge X \wedge \mu, \nu \rangle \\
& = &- \langle ||\nu || S_{\tilde \nu}(X) \wedge X \wedge \mu, \nu \rangle.
\end{eqnarray*}

\noindent Thus, we conclude that equation $(\ref{equation1})$ holds if and only if $X$ is an eigenvector of the $\tilde \nu$- shape operator. $\hfill\Box$

%\footnote{Equation \ref{equation1} 
%can be considered not only for vectors tangent to $M$, but even for vectors in $\mathbb R^4$ which might not be tangent to $M$.   
%What meaning could have the equation in this case?.}

\begin{rem}
Observe that this proposition allows us to consider any non-unitary normal vector field $\nu$ on $M$ in order
to define the $\tilde \nu$- principal directions of curvature. 
This fact will be used in the setting when we consider non-unitary gradient vector fields to define the lines of curvature.
\end{rem}

Consider now a surface $M$ defined as the transversal intersection of the inverse images of regular values of two differentiable functions,
there is a natural frame of the normal bundle of $M$, given by the gradient vector fields
of both functions. Thus, $M$ is endowed with a natural pair of lines of curvature.
A direct computation that uses the multilinearity properties of the left hand side of $(\ref{equation1})$
provides the following expression for the equation of these lines of curvature:

\begin{prop}\label{3.2}
Let $F,G:\mathbb R^4 \rightarrow \mathbb R$ be a pair of differentiable functions. Suppose that $M$ is a surface defined as the transversal intersection of
the inverse image of two regular values of these functions.
% restricted to $M$, namely: $M= F^{-1}(p) \cap G^{-1}(q)$, for some
%$p,\ q \in \mathbb R$  and each $x \in M$ is a regular point of both functions. 
Let $\nu$ and $\mu$ be the gradient vector fields of $F$ and $G$, respectively.
Then, the differential equation of the $\nu$-lines of curvature of $M$ is given by:

\begin{eqnarray}
\left( dx_{1} ,\ dx_{2},\  dx_{3},\ dx_{4} \right)
\left(\begin{array} {cccc}
    \Omega_{11} & \Omega_{12} & \Omega_{13} & \Omega_{14}  \\
     \Omega_{12} & \Omega_{22} & \Omega_{23} & \Omega_{24} \\
      \Omega_{13} & \Omega_{23} & \Omega_{33} & \Omega_{34} \\
\Omega_{14} & \Omega_{24} & \Omega_{34} & \Omega_{44} \\
\end{array} \right)\left(\begin{array}{c} dx_1 \\ dx_2 \\ dx_3 \\ dx_4 \end{array}\right)=0, \quad
\label{curvaturelines}
\end{eqnarray}

\noindent where $\{x_1,x_2,x_3,x_4\}$ is the system of parameters of $\mathbb R^4$ and 

%\footnote{este calculo aparece en 080907.1}

  $$
\Omega_{ij}= \frac{1}{2}\left(\langle \bar \nabla_{\frac{\partial}{\partial x_i}} \nu
\wedge \frac{\partial}{\partial x_j}\wedge \mu, \nu \rangle +
\langle \bar \nabla_{\frac{\partial}{\partial x_j}}
\wedge \frac{\partial}{\partial x_i} \wedge \mu, \nu \rangle\right).
  $$
\end{prop}

\begin{rem}
This proposition still holds if $p$, $q$ or both are singular values of any of the two functions, but only if $M= F^{-1}(p) \cap G^{-1}(q)$ 
lies in the complement of the 
corresponding critical points. The other hypothesis remain under assumption. 
\end{rem}

In case we have a frame of the tangent bundle of $M$, equation $( \ref{curvaturelines})$ has a simpler expression.

\begin{coro}\label{equationframe}
Assume that $V_1, V_2$ is a frame of the tangent bundle of $M$. Then, the expression of the differential equation $( \ref{curvaturelines})$ is:
\begin{eqnarray}
l_1^2 \Omega_{11} + l_1 l_2 \left( \Omega_{12} +  \Omega_{21} \right)+ l_2^2  \Omega_{22}=0
\end{eqnarray}

\noindent where $\Omega_{ij}=\langle \nabla_{V_i} \nu \wedge V_j \wedge \mu, \nu \rangle, \ \ i,j \in \{ 1,2\} $
and $l_i$ is the $i$-component of the vectorial solution $X$ of the equation, namely: $X= l_1V_1 + l_2V_2$.
\end{coro}

%\footnote{The proof of this corollary is in 300701}

%%%%%%%%%%%%%%%%%%%%%%%%%%%%%%%%%%%%%%%%%%%%%%%%%%%%%%%%%%%%%%%%%%%%%%%%%%%%%%%%%%%%%%%%%%%%%%%%%%%%%%%%%%%%%%%%%%%%%%%%%%%%%%%%

%%%%%%%%%%%%%%%%%%%%%%%%%%%%%%%%%%%%%%%%%%%%%%%%%%%%%%%%%%%%%%%%%%%%%%%%%%%%%%%%%%%%%%%%%%%%%%%%%%%%%%%%%%%%%%%%%%%%%%%%%%%%%%%%%

\section{An embbeding of the double torus into $\mathbb R^4$}

%%%%%%%%%%%%%%%%%%%%%%%%%%%%%%%%%%%%%%%%%%%%%%%%%%%%%%%%%%%%%%%%%%%%%%%%%%%%%%%%%%%%%%%%%%%%%%%%%%%%%%%%%%%%%%%%%%%%%%%%%%%%%%%%%

Let us analyze a family of examples where the $\nu$-lines of curvature can be described using the 
expressions presented above.  In \cite{Seade},
J. Seade studied the topology of a class of surfaces obtained as the projections of certain complex hypersurfaces
onto a real line through the origin of $\mathbb C$. Let us consider an element of that class:

  $$
\bar F:\mathbb C^2 \rightarrow \mathbb C,\ \ \bar F(z_1,z_2)= z_1^p + z_2^q,\ \ 1< p,q \in \mathbb N.
  $$

\noindent The hypersurface defined by the real part of $\bar F:\ Re\ \bar F= z_1^p+ z_2^q+ \bar z_1^p + \bar z_2^q$
is diffeomorphic to the cone over the link $L_r= Re\ \bar F^{-1}(0)\cap \mathbb S^3_{r},\ r>0$. Moreover, 
$L_r$ is a closed oriented surface of genus $(p-1)(q-1)$
in $\mathbb S^3_r \subset \mathbb R^4$. 
For instance, consider the case $p=2,q=3$ in which $L_r$ is a double torus in $\mathbb S^3_r \subset \mathbb R^4$. 
We provide an elementary proof of this fact in the following proposition. 

A direct computation
with the help of the standard identification between $\mathbb C^2$ and $\mathbb R^4$ given by
  $$
(z_1=x+iy,z_2=u+iv) \longmapsto (x,y,u,v)
  $$
\noindent shows that $Re\ \bar F$ has the following expression as a function defined in $\mathbb R^4$:

\begin{eqnarray}\label{FF}
F:\mathbb R^4 \rightarrow \mathbb R,\ \  F(x,y,u,v)=x^2- y^2 + u^3 - 3 u v^2.
\end{eqnarray}

\noindent Therefore, if we consider the function

\begin{eqnarray}\label{GG}
G_{r}:\mathbb R^4 \rightarrow \mathbb R,\ \ G_{r}(x,y,u,v)=x^2+ y^2 + u^2 + v^2 - r^2,
\end{eqnarray}

\noindent for any positive $r \in \mathbb R$, the intersection $F^{-1}(0) \cap G_{r}^{-1}(0)$ is the link $L_r$. Let us denote it by $T_{r}(2)$.

The transversality of the intersection is determined in terms of the gradient vector fields. Namely, 
let $F,G:\mathbb R^{4} \rightarrow \mathbb R$ be a pair of differentiable functions. Assume that 
$ p \in F^{-1}(a) \cap G^{-1}(b)$ is a regular point for both functions. The intersection of $ F^{-1}(a)$ and $G^{-1}(b)$
is transversal at $p$ if and only if the gradient vectors $\mu$ and $\nu$ of these functions are linearly independent at $p$. Thus, we state the following:

%\footnote{
%{\bf Lemma} 
%Let $F,G:\mathbb R^{N} \rightarrow \mathbb R$ be a pair of differentiable functions. Assume that
%$ p \in F^{-1}(a) \cap G^{-1}(b)$ is a regular point for both functions. The intersection of $ F^{-1}(a)$ and $G^{-1}(b)$
%is transversal at $p$ if and only if the gradient vectors $\mu$ and $\nu$ of these functions at $p$ are linearly independent.
%\noindent{\bf Proof of the lemma.}
%Suppose that $\mu$ and $\nu$ are linearly independent at $p \in F^{-1}(a) \cap G^{-1}(b)$, then the
%sum of the planes defined by these vector fields at $r$ generates $\mathbb R^N$. Since these planes are
%the tangent planes of the hypersurfaces which intersect at $p$, the intersection set is transversal. Conversely, if  we assume 
%that the vectors $\mu(p)$ and $\nu(p)$
%are linearly dependent the hyperplanes orthogonal to them coincide, and therefore their sum does not generate $\mathbb R^N$.
%This implies that the intersection is not transversal at $r$. $\hfill\Box$ }

\begin{prop}
The intersection $T_{r}(2)=F^{-1}(0) \cap G_{r}^{-1}(0), \  r>0$ is a smooth surface. 
\end{prop}

\noindent
{\bf Proof.}  We prove that the inverse images of zero $F^{-1}(0)$ and $G_{r}^{-1}(0)$  intersect transversally.
A direct computation shows that the origin is the unique singular point of both functions, $F$ and $G$. Since it does not lie in $G_{r}^{-1}(0)\ ,\ r>0,$ 
then it lies off $T_{r}(2)$. 
%\footnote{By applying lemma \ref{lemma 4.8}} 
Now, we show that the gradient vector fields $\mu(x,y,u,v)= (2x,-2y, 3u^2-3v^2,-6uv)$
and $\nu(x,y,u,v)=2(x,y,u,v)$ are linearly independent along $F^{-1}(0) \cap G_r^{-1}(0)$, if $r >0$.
To do so, we consider the equation:

\begin{eqnarray}
(2x,-2y, 3u^2-3v^2,-6uv) = \lambda(x,y,u,v),\ \lambda \in \mathbb{R}\setminus\{ 0 \} \label{5}
\end{eqnarray}

\noindent Suppose that $x \neq 0$. Then, $\lambda = 2$ and $y=0$. This implies
that $2v=-6uv$. Consequently, 
if $v \neq 0$, then $u= -\frac{1}{3}$ and the relation of the third coordinates above
implies that $2u= 3u^2 - 3v^2$, namely, $v^2=\frac{1}{3}$. By assuming that $G_{r}$ vanishes at $(x,0,-\frac{1}{3},\frac{1}{\pm \sqrt{3}})$, 
we get $x^2=r^2- \frac{4}{9}$. By supposing that $F$ vanishes at this point, 
%\footnote{we obtain that $r^2= \frac{4}{27}$ which implies that $x^2 =  \frac{-8}{27}$!.}
we get $x^2<0$ which is a contradiction.  So, if $x \neq 0$ then $v=0$ and $y=0$. 
This implies that $u=\frac{2}{3}$.
By evaluating $G_r$ at a point of the form $(x,0,\frac{2}{3},0)$, we get again a contradiction
%\footnote{Comparing the third coordinates we get $2u=3u^2$, as $u \neq 0,\ u=2/3$. So, $F(x,0,2/3,0)= x^2 + 8/27>0$ which implies that %the point is not on the locus of $F$}
.
Therefore, $x=0$. Analogous considerations derived from assuming that points of the form $(0,y,u,v)$ are in the intersection of the loci of $G_r$ and $F$ provide 
straightforward contradictions
%\footnote{If $y \neq 0$  then $\lambda =-2$ and comparing the 4th coordinates we get $v=3uv$. Two cases: i) $v \neq 0$ implies $u=1/3$, comparing the 3th coordinates we conclude $v^2= 1/3$. So evaluating $F(0,y, 1/3, \sqrt{1/3})<0$. Therefore, this point is not in the locus of $F$. ii) If $v=0$, since $y \neq 0$ then comparing the 3th coordinates we determine $u=-2/3$ and again $F(0,y,-2/3,0)<0$. Finally, if $x=y=0$ evaluating $F(0,0,u,v)=0$ we get $u^2=3v^2$. Comparing the 3th coordinates we get $2u^2$=-2u, i.e. $u=-1$. But comparing the 4th coordinates we get $-2v=6v$. That is $v=0$. So the unique solution is the origin}
. Therefore, the gradient vector fields are linearly independent at any point of $T_r(2)$.
$\hfill\Box$

Now, we consider the symmetries of $T_{r}(2)$.
It is straightforward to verify that this surface is invariant under the action of the group
$\mathcal G$ generated by the following applications:
%\footnote{RegionFundamental311016.pdf}
 
$$\Gamma_i: \mathbb C^2 \rightarrow \mathbb C^2,\ \ \ i=1,2,3,$$  
 $\Gamma_1(z_1,z_2)=(\bar z_1, z_2)$, $\Gamma_2(z_1,z_2)=(-\bar z_1, z_2)$ and $\Gamma_3(z_1,z_2)=(z_1, e^{\frac{2\pi i}{3}} z_2)$.

\medskip

Let us describe this embedding of the double torus. 

%\footnote{Compute de Gaussian curvature...not constant?}

From the very definition of the functions $F$ and $G$ we have   
four charts of $T_r(2)$ given by the projection onto the $uv$-plane of the following hexagonal regions. 
%\footnote{T2Parametrization150916}

\begin{eqnarray}\label{regions}
H_{++} & = & \lbrace (x,y,u,v) \in T_r(2) : x \geq 0, y \geq 0 \rbrace,  \nonumber \\ 
H_{+-} & = & \lbrace (x,y,u,v) \in T_r(2) : x \geq 0, y \leq 0 \rbrace, \\
H_{-+} & = & \lbrace (x,y,u,v) \in T_r(2) : x \leq 0, y \geq 0 \rbrace, \nonumber \\
H_{--} & = & \lbrace (x,y,u,v) \in T_r(2) : x \leq 0, y \leq 0 \rbrace. \nonumber
\end{eqnarray}

We consider these sets with boundary, since it will be useful to identify corresponding boundaries in the construction below.
The projection $\mathbb{R}^{4} \rightarrow \mathbb{R}^{2}$
\[ (x,y,u,v) \mapsto (u,v) \] defines the hexagon $H_{uv}$ on the $uv$-plane as follows. 

\[ H_{uv}=\lbrace (u,v) \in \mathbb{R}^{2} \quad :  \quad 
r^{2}-u^{3}+3uv^{2}-u^{2}-v^{2} \geq 0, \quad r^{2}+u^{3}-3uv^{2}-u^{2}-v^{2} \geq 0 \rbrace. \]

Thus, we consider the coordinate functions 

$$\varphi_{\pm  \pm}=\left(\pm \frac{\sqrt{r^2-u^2-u^3-v^2+3 u v^2}}{\sqrt{2}}, \pm \frac{\sqrt{r^2- u^2 + u^3- v^2 - 3 u v^2}}{\sqrt{2}},u, v \right),$$
where the domain of these applications is $H_{uv}$. 
%\footnote{
%Let $H_{uv}$ denote this set. Note that the equations defining $T_r(2)$ can be expressed as
%\[ 2x^{2} = r^{2}-u^{3}+3uv^{2}-u^{2}-v^{2} \]
%\[ 2y^{2} = r^{2}+3u^{3}-3uv^{2}-u^{2}-v^{2}. \]
%Then $H_{uv}$ can be described as follows:
%\[ H_{uv}=\lbrace (u,v) \in \mathbb{R}^{2} \quad :  \quad r^{2}-u^{3}+3uv^{2}-u^{2}-v^{2} \geq 0, 
%\quad r^{2}+u^{3}-3uv^{2}-u^{2}-v^{2} \geq 0 \rbrace \]}
%In the figure 
 
\begin{figure}[htb]
\begin{center}
\includegraphics[width=7cm]{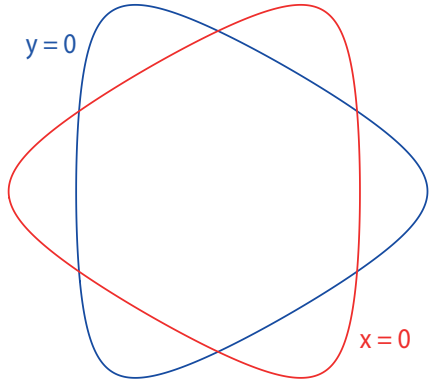}
\end{center}
\caption{Projection of the double torus on the $uv$-plane.}
\label{Hexagon}
\end{figure}

%[width=5.5cm]{hexagono1.eps} 
The boundary of $H_{uv}$ is constituted by 
\[ \partial H_{uv} = X^{1} \cup Y^{1} \cup X^{2} \cup Y^{2} \cup X^{3} \cup Y^{3}, \]
where

\begin{eqnarray*}
X^{1} & = & \lbrace (u,v) \quad : \quad r^{2}-u^{3}+3uv^{2}-u^{2}-v^{2}=0, \quad u \geq 0 \rbrace, \\
Y^{1} & = & \lbrace (u,v) \quad : \quad r^{2}+u^{3}-3uv^{2}-u^{2}-v^{2}=0, \quad u \geq 0, \quad v \geq 0 \rbrace, \\
X^{2} & = & \lbrace (u,v) \quad : \quad r^{2}-u^{3}+3uv^{2}-u^{2}-v^{2}=0, \quad u \leq 0, \quad v \geq 0 \rbrace, \\
Y^{2} & = & \lbrace (u,v) \quad : \quad r^{2}+u^{3}-3uv^{2}-u^{2}-v^{2}=0, \quad u \leq 0 \rbrace, \\
X^{3} & = & \lbrace (u,v) \quad : \quad r^{2}-u^{3}+3uv^{2}-u^{2}-v^{2}=0, \quad u \leq 0, \quad v \leq 0 \rbrace, \\
Y^{3} & = &\lbrace (u,v) \quad : \quad r^{2}+u^{3}-3uv^{2}-u^{2}-v^{2}=0, \quad u \geq 0, \quad v \leq 0 \rbrace.
\end{eqnarray*}

%where the domains of these applications are the projections on the hexagons, respectively. 

%\footnote{toro-doble-agosto-2016.pdf}

\noindent Observe that

\[ T(2)_+ =  T_r(2) \cap \lbrace x \geq 0 \rbrace = H_{++} \cup H_{+-} \]
\[ T(2)_- =  T_r(2) \cap \lbrace x \leq 0 \rbrace = H_{-+} \cup H_{--}. \]

Therefore, $T(2)_+$ is constructed by gluing the closure of two hexagonal regions through the identification of the curves defined by the equation $y=0$. 
That is, $T(2)_+$ is a surface homeomorphic to the sphere minus tree discs (see figure $2$).

\begin{figure}[htb]
\begin{center}
\includegraphics[width=7cm]{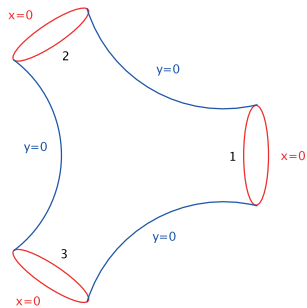}
\end{center}
\caption{First identification.}
\label{Esferatreshoyos}
\end{figure}

Since
$T(2)_-$ is the image of $ T(2)_+$ by the reflection with respect to the hyperplane $x=0$, we obtain
$T(2)$ by gluing two surfaces of this type along the boundary of the missing discs. 

\begin{figure}[htb]
\begin{center}
\includegraphics[width=7cm]{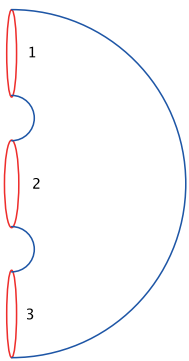}
\end{center}
\caption{A half of a double torus.}
\label{Hexagon}
\end{figure}

\begin{rem}\label{conjugaciontraslacion}
Let us denote by $T_{\pm \pm}: \mathbb C^2 \rightarrow \mathbb C^2$, the translation of the origin to any of the points $\varphi_{\pm  \pm}(0,0)$, respectively.
Then, the conjugation $T_{\pm \pm}^{-1}\circ g  \circ T_{\pm \pm}$, where $g$ belongs to the group generated by $\Gamma_3$, leaves $T(2)$ invariant.  
%\footnote{ver pie de p\'agina anterior}
\end{rem}

We describe the action of $\mathcal G$ on $T_r(2)$. Each point in the interior of $H_{uv}$ represents 4 points identified by the action of the 
subgroup of reflections of $\mathcal G$ on the torus. 
We denote by $I$ the set of the intersections of the curves $X_i$ and $Y_j$, $i,j = 1,2,3.$
Each point in the boundary $H_{uv}\setminus I $ represents 2 points, while, points in $I$ represent only one point of $T_r(2)$. 

Observe that it is enough to determine this region only in $H_{++}$, since 
any of the hexagonal regions described in $(\ref{regions})$ are obtained from this one by one reflection.

We analyze the action of $\Gamma_3$ on the closure of the hexagonal region $H_{++}$. So,
consider the following subset of $H_{++}$
\begin{eqnarray}\label{pentagono}
P_{++} = \lbrace (x,y,u,v) \in H_{++} | v \geq 0, v \geq - \sqrt{3} u \rbrace. 
\end{eqnarray}
This subset is the closure of a pentagonal region.
Let $P_0$ be the image of the projection of $P_{++}$ in the plane $uv$ (see figure \ref{Pentagon}). We describe the boundary of $P_0$ as follows:
The line segment which goes from the origin to the midpoint of $X^{1}$ is denoted by $e_{1}$;
the upper half of $X^{1}$ is denoted by $e_{2}$; the edge $Y^{1}$ is denoted by $e_{3}$;
the upper half of $X^{2}$ is denoted by $e_{4}$ and finally; 
$e_{5}$ denotes the line segment which goes from the midpoint of $X^{2}$. 

\begin{figure}[htb]
\begin{center}
\includegraphics[width=7cm]{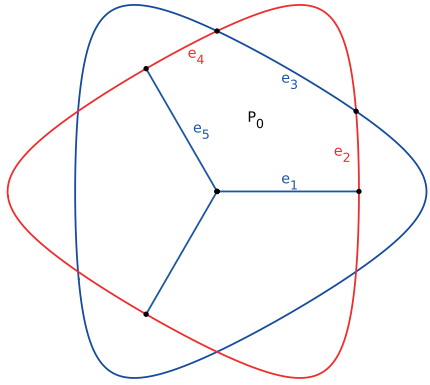}
\end{center}
\caption{A pentagonal fundamental region}
\label{Pentagon}
\end{figure}

If we rotate this pentagon $P_{0}$ by $2\pi/3 $ we obtain another pentagon $P_{1}$, and if we rotate $P_{1}$ we obtain a third pentagon $P_{2}$. 
The closure of the projection of $H_{++}$ is the union of closure of these three pentagons. We can define analogously, pentagonal regions in 
$P_{-+},\ P_{+-}$ and $P_{--}$ contained in the other hexagonal regions.
%So, $R_{0}$ is like a fundamental region, because the union of its images cover $T(2)$, but it is not a fundamental region (the points on the border do not %satisfy the definition). The important think is that we can use the symmetry group to reduce the analysis to the points in $R_{0}$ without loss of generality.

%So consider the convex region in the uv-plane defiend by 

%$(u,v)| v \geq 0,$  

\section{Lines of curvature of the double torus} 
%\pagebreak

\subsection{The lines of curvature of the double torus in $\mathbb R^4$}

We consider the unitary vector field 
  
$$ \nu= \frac{F}{||F||},\ {\rm where}\ 
\nabla F=(2 x, -2 y, 3 u^2 - 3 v^2, -6 u v)
  $$
is the gradient vector field of $F$
on $T_{r}(2)$.

Proposition \ref{3.2} provides the differential equation of $\nu$-lines of curvature
of the double torus $T_r(2)$ immersed in $\mathbb S^3 \subset \mathbb R^4$.

\begin{coro}\label{EcuaDiff}
The differential equation of the $\nu$-lines of curvature of $T_r(2)$ has the following expression:
{\small
\begin{eqnarray}
dX
\left(\begin{array} {cccc}
    0 & 6v(v^2-3u^2) & -yv(2-18u+9\rho^2) & y(9(u^2-v^2)+u(2+9\rho^2))  \\
     * & 0 & xv(2+18u+9\rho^2) & x(9(u^2-v^2)-u(2+9\rho^2)) \\
      * & * & -24xyv & -24xyu \\
* & * & * & 24xyv\\
\end{array} \right)dX^T=0, \quad
\end{eqnarray}}

\noindent where $dX=\left( dx ,\ dy,\  du,\ dv \right)$, $\rho^2=u^2+v^2$ and the
values of $*$ are determined by the symmetry of the matrix.

\end{coro}

%%%%%%%%%%%%%%%%%%%%%%%%%%%%%%%%%%%%%%%%%%%%%%%%%%%%%%%%%%%%%%%%%%%%%%%%%%%%%%%%%%%%%%%%%%%%%%%%%%%%%%%%%%%%%%%%%%%%%%%%%
%%%%%%%%%%%%%%%%%%%%%%%%%%%%%%%%%%%%%%%%%%%%%%%%%%%%%%%%%%%%%%%%%%%%%%%%%%%%%%%%%%%%%%%%%%%%%%%%%%%%%%%%%%%%%%%%%%%%%%%%%

We provide the description of the foliation of $\nu$-lines of curvature of the double torus
$T_r(2),\ r>0$.
As is common, we consider {\it the index of an isolated $\nu$-umbilic point} as the index of an isolated singularity 
of the field of principal directions on the surface, \cite{Hopf}.  
According to the classification provided in \cite{Soto}, (see also \cite{Bruce}) 
we refer to a {\it $\nu$-umbilic point of type $D_3$} as a generic $\nu$-umbilic point with three separatrices and index $-\frac{1}{2}$. 
We define a {\it pathwise separatrix of the $\nu$-principal foliation} as a smooth curve constituted by the union 
of separatrices, some of them of maximal $\nu$-principal curvature and the other of minimal $\nu$-principal curvature.    
In the theorem below, we will omit the prefix $\nu$ for the elements of the foliation of $\nu$-lines of curvature, such as 
$\nu$-umbilic points, $\nu$-separatrices, etc.  

\begin{teo}\label{Maintheorem}
The foliation of $\nu$-lines of curvature of the double torus $T_r(2)$ has the following structure: 

\noindent i) It only has 4 umbilic points, each one of type $D_3$.

\noindent ii) It has $3$ pathwise separatrices; each one is an immersion of $\mathbb S^1$ 
%\footnote{...is homeomorphic to a circle} 
containing the whole set of umbilics.

\noindent iii) The first homology group of the double torus $T_r(2)$ is generated by
1-cycles that are pathwise separatrices.

\noindent iv) The complement of the closure of the set of separatrices is foliated by cycles, that is, closed curvature lines diffeomorphic 
to $\mathbb S^1$.  

\end{teo}
%\footnote{include a picture}

%%%%%%%%%%%%%%%%%%%%%%%%%%%%%%%%%%%%%%%%%%%%%%%%%%%%%%%%%%%%%%%%%%%%%%%%%%%%%%%%%%%%%%%%%%%%%%%%%%%%%%%%%%%%%%%%%%%%%%%%%%
%%%%%%%%%%%%%%%%%%%%%%%%%%%%%%%%%%%%%%%%%%%%%%%%%%%%%%%%%%%%%%%%%%%%%%%%%%%%%%%%%%%%%%%%%%%%%%%%%%%%%%%%%%%%%%%%%%%%%%%%%%

We begin by proving statement $i)$ in the following two propositions.

\begin{prop}\label{4.1}
Let $r \in \mathbb R$ such that $0<r<\frac{1}{\sqrt{10}} $.
The foliation of $\nu$-lines of curvature of $T_r(2)$ has 4 umbilic points.
\end{prop}

\noindent{\bf Proof.} Let us consider the open subset $T_r(2)^1= \{(x,y,u,v)\in T_r(2)| xy \neq 0  \}$ of $T_r(2)$. We define on $T_r(2)^1$ the frame:

\begin{eqnarray*}
V_1^1(x,y,u,v) & = & -\left(\frac{v-3uv}{2x}\right)\frac{\partial}{\partial x}-
\left(\frac{v+3uv}{2y}\right)\frac{\partial}{\partial y}+
\frac{\partial}{\partial v}, \\
V_2^1(x,y,u,v) & = & -\left(\frac{2u+ 3u^2-3v^2}{4x}\right)\frac{\partial}{\partial x}-
\left(\frac{2u- 3u^2+3v^2}{4y}\right)\frac{\partial}{\partial y}+
\frac{\partial}{\partial u}.
\end{eqnarray*}

\noindent After multiplying 
the coefficients of the differential equation in Theorem $\ref{equationframe}$ of $\nu$-lines of curvature on $T_r(2)^1$ by the monomial $xy$,
they have the expression:
%\footnote{despues de multiplicar por un factor $xy$, ver 090718LCG07.nb}:

\begin{eqnarray*}
\Omega_{11}&=& v (-18 u^3 (x^2 - y^2) + 2 u (1 + 18 v^2) (x^2 - y^2) + 
   27 u^4 (3 v^2 + x^2 + y^2) - \\
  && 3 u^2 (9 v^4 + x^2 + y^2 + v^2 (3 - 9 x^2 - 9 y^2)) + 
   3 (v^4 + 8 x^2 y^2 + 3 v^2 (x^2 + y^2)))\\
\Omega_{12}&=& \frac{1}{2}(36 u^4 (x^2 - y^2) - 54 u^2 v^2 (x^2 - y^2) + 
 2 v^2 (-2 + 9 v^2) (x^2 - y^2) \\
&&- 27 u^5 (3 v^2 + x^2 + y^2) + 
 6 u^3 (-3 v^2 + 18 v^4 - 2 (x^2 + y^2)) - \\
&& 3 u (9 v^6 + 16 x^2 y^2 + 4 v^2 (x^2 + y^2) - 
    v^4 (2 + 9 x^2 + 9 y^2)) )\\
\Omega_{21}&=&\frac{1}{2}(-81 u^5 v^2 + 9 u^4 (x^2 - y^2) - 4 u^2 (-1 + 27 v^2) (x^2 - y^2)+ \\
&& 9 v^4 (-x^2 + y^2) + 
 6 u^3 (18 v^4 - 2 (x^2 + y^2) - 3 v^2 (1 + 3 x^2 + 3 y^2)) - \\
&& 3 u (9 v^6 + 16 x^2 y^2 + 4 v^2 (x^2 + y^2) + 
    2 v^4 (-1 + 9 x^2 + 9 y^2)))\\
\Omega_{22}&=& -\frac{v}{4} (-81 u^6 + 9 u^4 (4 + 21 v^2 - 6 x^2 - 6 y^2) - 
   72 u^3 (x^2 - y^2) +\\
&& 8 u (1 + 18 v^2) (x^2 - y^2) - 
 3 u^2 (4 v^2 + 45 v^4 - 20 (x^2 + y^2)) + 3 (9 v^6 + \\
&& 32 x^2 y^2 + 
4 v^2 (x^2 + y^2) + 18 v^4 (x^2 + y^2))). 
\end{eqnarray*}

Taking into account that conditions $x^2+y^2=-(u^2+v^2) + r^2$ and $x^2-y^2=-(u^3-3uv^2)$ hold when restricted to $T_r(2)$,
these coefficients can be expressed only in terms of $u,v$ as follows:

\begin{eqnarray}\label{sys5}
\Omega_{11}& = & v (6 r^4 - 15 u^6 + 3 u^2 v^2 + u^4 (7 - 27 v^2) +  
  3 r^2 (9 u^4 - v^2 + \nonumber \\ && u^2 (-5 + 9 v^2))), \nonumber \\ 
\Omega_{12}+ \Omega_{21} & = & \frac{u}{2} (-24 r^4 - 3 r^2 (9 u^4 + v^2 (-8 + 9 v^2) + 2 u^2 (-4 + 9 v^2))  \\
 &&+ \ 2 u^2 (3 u^4 + v^2 (-10 + 9 v^2) + u^2 (-2 + 36 v^2))), \nonumber \\ 
\Omega_{22} & = & \frac{v}{4} (-24 r^4 - 21 u^6 + u^4 (8 - 27 v^2) - 3 u^2 v^2 (-4 + 9 v^2) + \nonumber \\ &&
   3 v^4 (-4 + 9 v^2) + 6 r^2 (-2 u^2 + 9 u^4 + 6 v^2 - 9 v^4)).  \nonumber  
\end{eqnarray}

\noindent Therefore, points in $T_r(2)$ where $u=v=0$ are $\nu$-umbilic, namely, points in the intersection
of the plane curves $x^2-y^2=0$ and $x^2+y^2- r^2=0$. These points are:
  $$
\frac{r}{\sqrt 2}(1,1,0,0),\ \frac{r}{\sqrt 2}(-1,1,0,0),\
\frac{r}{\sqrt 2}(1,-1,0,0),\frac{r}{\sqrt 2}(-1,-1,0,0).   $$

A simple analysis, which is omitted, shows that there are no umbilic points in the subset $T_r(2)\setminus T_r^1(2)$.   

$\hfill\Box$
\begin{prop}\label{teoindex}
The foliation of $\nu$-lines of curvature of $T_r(2)$ only has 4 umbilic points.
\end{prop}

\noindent {\bf Proof.}
To prove this property we will fix the radius $r=\frac{1}{\sqrt{10}}$ in order to get convenient computations.
Once we prove the property for the double torus $ T_{\frac{1}{\sqrt{10}}}(2)$, the general case 
follows the fact that the
$\nu$-principal foliations are invariant under homotheties of $\mathbb R^4$ centered at the origin. 

The $\nu$-umbilic points determined in Proposition \ref{4.1} in this case are:
 
$$
\left \{\frac{1}{\sqrt{20}}(1,1,0,0), \frac{1}{\sqrt{20}}(-1,1,0,0), \frac{1}{\sqrt{20}}(1,-1,0,0),\frac{1}{\sqrt{20}}(-1-,1,0,0) \right\}.
$$

Now we show that these points are the unique $\nu$-umbilics on $T_{\frac{1}{\sqrt{10}}}(2)$.

%%%%%%%%%%%%%%%%%%%%%%%%%%%%%

%For sake of clearness we present the strategy of the proof. 
We assume that $uv \neq 0$ and denote the polynomials 
$\frac{1}{v}\Omega_{11}$, $\frac{1}{u}(\Omega_{12}+ \Omega_{21})$ and $\frac{1}{v}\Omega_{22}$, by $\Omega_1$, $\Omega_2$, $\Omega_3$, respectively. 
Namely,

\begin{eqnarray}\label{omega}
\Omega_1 &=& 6 r^4 - 15 u^6 + 3 u^2 v^2 + u^4 (7 - 27 v^2) +  
  3 r^2 (9 u^4 - v^2 + u^2 (-5 + 9 v^2)), \nonumber \\
\Omega_2 &=&  \frac{1}{2} (-24 r^4 - 3 r^2 (9 u^4 + v^2 (-8 + 9 v^2) + 2 u^2 (-4 + 9 v^2)) + \\
 &&  2 u^2 (3 u^4 + v^2 (-10 + 9 v^2) + u^2 (-2 + 36 v^2))), \nonumber \\ 
\Omega_3 &=& \frac{1}{4} (-24 r^4 - 21 u^6 + u^4 (8 - 27 v^2) - 3 u^2 v^2 (-4 + 9 v^2) + \nonumber \\ &&
   3 v^4 (-4 + 9 v^2) + 6 r^2 (-2 u^2 + 9 u^4 + 6 v^2 - 9 v^4)).  \nonumber
\end{eqnarray}

Straightforward considerations show that $\nu$-umbilic points satisfy the assumption.
Thus, $p \in T_{\frac{1}{\sqrt{10}}}(2)$
is umbilic if and only if $p$ lies in the intersection of the zero loci of $\Omega_i,\ i=1,2,3.$ 
%The goal is to prove that this condition is only satisfied 
%by the points described above. 
Consequently, to determine this intersection we define a reduction
of these polynomials regarding the following lemma whose proof is straightforward, and is therefore omitted. 

\begin{lem}\label{lema1} 
Let $\rho: \mathbb R^n \rightarrow \mathbb R^m,\ (x_1,...,x_n) \mapsto (y_1,...y_m)$ be a polynomial mapping.
Assume that the induced application
$\rho ^*: \mathbb R[y_1,...y_m] \rightarrow \mathbb  R[x_1,...x_n] $, defined by $\rho^*(Q)= P$, where $P=Q \circ \rho$,
implies that $\rho({\mathcal L}(P)) \subset {\mathcal L}(Q)$, where ${\mathcal L}(P)$ and ${\mathcal L}(Q)$ are the loci of $P$ and $Q$, respectively.
\end{lem}

%\footnote{proof:  if $q \in \rho({\mathcal L}(P))$, $q=\rho(p)$ for some $p \in \rho({\mathcal L}(P))$. 
%Thus, $0= P(p)= Q\circ \rho(p)= Q(q)$. Then $q \in {\mathcal L}(Q)$.} 

Let 
$$\rho:\mathbb R^2 \rightarrow \mathbb R^2, \ \ 
(u,v) \mapsto (w=u^2,z=v^2).$$
  
Thus, the polynomials

\begin{eqnarray*}
\Theta_1(w,z) & = &\frac{3}{50} - \frac{3w}{2} + \frac{97 w^2}{10} - 15 w^3 - \frac{3 z}{10} + \frac{57 w z}{10} - 
 27 w^2 z,  \\
\Theta_2(w,z) & = &-\frac{6}{25} - \frac{6 w}{5} + \frac{67 w^2}{5} - 21 w^3 + \frac{18 z}{5} + 12 w z - 
 27 w^2 z - 
\\ &&\frac{87 z^2}{5} - 27 w z^2 + 27 z^3, \\
\Theta_3(w,z) & = &-\frac{6}{25} + \frac{12 w}{5} - \frac{67 w^2}{10} + 6 w^3 + \frac{12 z}{5} - \frac{127 w z}{5} + 
 72 w^2 z -  \\ 
&& \frac{27 z^2}{10} + 18 w z^2,
\end{eqnarray*}
satisfy $\rho ^*(\Theta_i)= \Omega_i,\ i=1,2,3$.

Since Lemma \ref{lema1} implies that $\bigcap_{i=1}^3 \rho({\mathcal L}(\Omega_i)) \subset \bigcap_{i=1}^3{\mathcal L}(\Theta_i)$,
it is enough to show that $\bigcap_{i=1}^3{\mathcal L}(\Theta_i)= \emptyset$ to complete the proof
%\footnote{Indeed, if $p \in \bigcap_{i=1}^3 ({\mathcal L}(\Omega_i))$ then, $\rho(p) \in {\mathcal L}(\Theta_i) \forall i$}
.
We fix the radio of the sphere and compute using Wolfram Mathematica to get a set of generators of the ideal defined by $\Theta_i,\ i=1,2,3.$ Then, we state

\begin{lem} Let us asumme that $r=\frac{1}{\sqrt{10}}$. Then,
the ideal generated by the polynomials $\Theta_i, i=1,2,3$ has the following Groebner basis:
\Small{ 
\begin{eqnarray*}
G_1(w,z)&=& \Gamma_1(z)=2457 + 355500 z - 16427260 z^2 + 227180200 z^3 - 1418476000 z^4 + 
 \\
&&  4524012000 z^5 - 7251120000 z^6 + 4665600000 z^7 \\ 
G_2(w,z)&=& -272440432411875 + 651925265295213 w + 3881340837877779 z + \\
&& 6558340381939640 z^2 - 389422787482597000 z^3 + 
 2318821632923662800 z^4  \\ 
&& - 5434755866556408000 z^5 + 
 4628749229159040000 z^6. 
\end{eqnarray*}}
\end{lem}

%\footnote{see 090718LCG07.nb} 

The polynomial $G_1$ depends only  on $z$ and its  roots are computable. The positive ones are

\begin{eqnarray*}
z_1& =& \frac{1}{20}, \ \  {\rm and} \\   
z_2& = &\frac{1}{240}\left(67 - \left(\frac{13^5}{\sigma}\right)^{\frac{1}{3}} - (13 \sigma )^{\frac{1}{3}}\right),
\end{eqnarray*} 
where $\sigma=1289-216 \sqrt{35}$.
Thus, the intersection of the loci of $G_1$ and $G_2$ is obtained by substituting these roots in the polynomial
$G_2$ and solving for $w$. That is, for $z_1$ we get 
$w_1=3/20$, while for 
$z_2$ we get 
\begin{eqnarray*}
w_2&=&\frac{\sigma^{-\frac{5}{3}}}{720}(13^{\frac{1}{3}}(556848 \sqrt{35}-3294481 )+13^{\frac{2}{3}}\sigma^{\frac{1}{3}}(2808 \sqrt{35}-16757) - \\
&& \sigma^{\frac{2}{3}}(14472 \sqrt{35}-86363 )).
\end{eqnarray*}
\medskip
 By a direct substitution we see that none of these points lies on the sphere of radius $r=\frac{1}{\sqrt{10}}$. This completes the proof under the hypothesis 
$xy \neq 0$. If $x$ or $y$ vanishes, a similar analysis can be applied. It is simpler because most of the equations above involve easier expressions, so it is omitted.
$\hfill\Box$ 

%%%%%%%%%%%%%%%%%%%%%%%%%%%%%%%%%%%%%%%%%%%%%%%%%%%%%%%%%%%%%%%%%%%%%%%%%%%%%%%%%%%%%%%%%%%%%%%%%%%%%%%%%%%%%%%
\medskip

Let us analyze the $\nu$-lines of curvature near the umbilic points in order to determine the umbilic separatrices and their indexes.

%\footnote{See Chart[2]+Separatrices.nb for details}
Taking the umbilic point $\dfrac{r}{\sqrt 2}(1,1,0,0)$ as the origin, we can introduce, for the local analysis, a system of coordinates $X,Y,Z,W,$ in a neighborhood of this
point, where the double torus can be parameterized as the graph of a differentiable function $h:U \subset \mathbb R^2\to \mathbb R^2.$ The parametrization has the form
${\bf x}(u,v)=(f(u,v),g(u,v),u,v),$ where:

\begin{align*}
  f(u,v)=&\dfrac{r}{\sqrt 2}-\dfrac{\sqrt{r^2-u^2-u^3-v^2+3uv^2}}{\sqrt 2},\\
  g(u,v)=& \dfrac{r}{\sqrt 2}- \dfrac{\sqrt{r^2-u^2+u^3-v^2-3uv^2}}{\sqrt 2}.
\end{align*}

\begin{lem}
The separatrices at the umbilic point $\dfrac{r}{\sqrt 2}(1,1,0,0)$ have slope $p=0,\;\pm \sqrt 3.$
\end{lem}
\noindent {\bf Proof.}
In \cite{F} is shown that if we have a local immersion of a surface into $\mathbb R^4$ with a parameterization of the form
\begin{align*}
  {\bf x}(u,v)&=( \frac{k}{2}(u^2+v^2)+\frac{a}{6}u^3+\frac{b}{2}uv^2+\frac{c}{6}v^3+R_1,\\
  &\frac{\alpha}{2}u^2+\beta uv+ \frac{\gamma}{2}v^2+\frac{\delta}{6}u^3+\frac{\epsilon}{2}u^2v+\frac{\zeta}{2}uv^2+\frac{\eta}{6}v^3+R_2,u,v),
\end{align*}
then the slopes of the separatrices at the umbilic point are given by the solutions of the equation:
  $$
(b+\beta n)p^3+(\beta m-c-(\gamma-\alpha)n)p^2-(2b-a+(\gamma-\alpha)m+\beta n)p-\beta m=0,
  $$
where $p=dv/du$ is the variable to determine, 
and $k,\;a,\;b,\;c,\;\alpha,\;\beta,\;\gamma,\;\delta,\;\epsilon,\;\zeta,\;\eta$ are the local parameters of the immersion.

We use the Taylor expansion of $f(u,v)$ and $g(u,v)$ up to order three, to express the parameterization as follows:
\begin{align*}
  {\bf x}(u,v)=& (\dfrac{1}{2\sqrt 2r}(u^2+v^2)+\dfrac{3}{6\sqrt 2 r}u^3-\dfrac{3}{2\sqrt 2r}uv^2,\\
  & \dfrac{1}{2\sqrt2r}(u^2+v^2)-\dfrac{1}{2\sqrt2r}u^3+\dfrac{9}{6\sqrt2r}uv^2,u,v).
\end{align*}
By solving the previous equation we obtain that $p=0,\;\pm \sqrt 3.$ $\Box$

Then, by applying the action of the group $\mathcal G$ on $T_r(2)$ we conclude the following: 

\begin{prop}\label{teoindex}
The umbilic points of $T_r(2)$ are of type $D_3$.
\end{prop}

\medskip
Now we prove statement ii).

\begin{prop}\label{circ1}
  There exists a pathwise separatrix, homeomorphic to a circle, which goes through all the umbilic points.
\end{prop}
\begin{proof}
If we consider a curve $S_1(t):I\to \mathbb R^4$ given by $S_1(t)=(h_1(t), h_2(t),$ $h_3(t), h_4(t)),$ it is easy to show that it satisfies the
equation of the curvature lines when the last coordinate vanishes, i.e., for $h_4(t)=0.$ Also, if we consider that the curve is in the double torus,
 $S_1(t)$ should simultaneously satisfy conditions $F(S_1)=0$ and $G_r(S_1)=0$. Consequently, we obtain a curve defined in $4$ parts given by:
\begin{align*}
  S_{11}(t)&=(-\frac{\sqrt{r^2-t^2 -t^3}}{\sqrt 2},-\frac{\sqrt{r^2-t^2 + t^3}}{\sqrt 2},t,0),\\
  S_{12}(t)&=(\frac{\sqrt{r^2-t^2 -t^3}}{\sqrt 2},-\frac{\sqrt{r^2-t^2 + t^3}}{\sqrt 2},t,0),\\
  S_{13}(t)&=(\frac{\sqrt{r^2-t^2 -t^3}}{\sqrt 2},\frac{\sqrt{r^2-t^2 + t^3}}{\sqrt 2},t,0),\\
  S_{14}(t)&=(-\frac{\sqrt{r^2-t^2 -t^3}}{\sqrt 2},\frac{\sqrt{r^2-t^2 + t^3}}{\sqrt 2},t,0).
\end{align*}
We have that $U_i \in S_{1i}(t),\;i=1,2,3,4$ where $U_1=(-\frac{r}{\sqrt 2},-\frac{r}{\sqrt 2},0,0),$
$U_2=(\frac{r}{\sqrt 2},-\frac{r}{\sqrt 2},0,0),$ $U_3=(\frac{r}{\sqrt 2},\frac{r}{\sqrt 2},0,0),$ $U_4=(-\frac{r}{\sqrt 2},\frac{r}{\sqrt 2},0,0)$
%\footnote{Ver T2separatrices251012.nb, adem\'as en este calculo se usa la formulacion de las equaciones en forma impl\'icita, quiz\'a tambi\'en se pueda ver 
%con la parametrizaci\'on de la carta, pero no es claro... revisar mas adelante.}
.

We have that $(\Gamma_1\circ \Gamma_2 \circ \Gamma_1\circ \Gamma_2)(S_{11})= (\Gamma_1\circ \Gamma_2 \circ \Gamma_1)(S_{12})=(\Gamma_1\circ \Gamma_2)(S_{13})=\Gamma_1(S_{14})=S_{11}.$ In addition,
 the matrices $\Gamma_2$ and $\Gamma_1$ have as fixed points the points obtained taking $x=0$ for $\Gamma_2$ and $y=0$ for $\Gamma_1.$
 Taking this in the curvature line, we obtain $4$ fixed points $p_i,\; i=1,2,3,4$ where $p_1=S_{12}\cap S_{13},$
$p_2=S_{11}\cap S_{14},$ $p_3=S_{11}\cap S_{12}$ and  $p_4=S_{13}\cap S_{14}$.
%\footnote{It is convenient to use the matrix representation of the operators $\Gamma_i}
Thus, since $S_1(t)$ is a differentiable curve and $S_{1i},\;i=1,2,3,4$ are connected arcs with just one point in common each two of them,
and since we can generate $S_{1i},\; i=2,3,4$ with the action of $\mathcal G$ on $S_{11}$, we conclude that  $S_1(t)$ is homeomorphic to a circle. 
\end{proof}

\begin{prop}\label{circ2}
There exist two pathwise separatrices obtained by rotating $S_1(t)$ $2\pi/3$ and $4\pi/3$ with center at an umbilic point.
\end{prop}
\begin{proof} We know that the curvature lines are invariant under rotations. By applying $\Gamma_3$ to the separatrix $S_1$, 
we obtain a new curvature line $S_2:I\to \mathbb R^4$ associated to $p=\sqrt 3$ given by:

\begin{align*}
  S_{21}(t)&=(-\frac{\sqrt{r^2-t^2 - t^3}}{\sqrt 2}, -\frac{\sqrt{r^2-t^2+t^3}}{\sqrt 2},-\frac{t}{2},\frac{\sqrt 3}{2}t),\\
  S_{22}(t)&=(\frac{\sqrt{r^2-t^2 - t^3}}{\sqrt 2}, -\frac{\sqrt{r^2-t^2 + t^3}}{\sqrt 2},-\frac{t}{2},\frac{\sqrt 3}{2}t),\\
  S_{23}(t)&=(\frac{\sqrt{r^2-t^2-t^3)}}{\sqrt 2}, \frac{\sqrt{r^2- t^2 + t^3}}{\sqrt 2},-\frac{t}{2},\frac{\sqrt 3}{2}t),\\
  S_{24}(t)&=(-\frac{\sqrt{r^2-t^2-t^3}}{\sqrt 2}, \frac{\sqrt{r^2-t^2+t^3}}{\sqrt 2},-\frac{t}{2},\frac{\sqrt 3}{2}t),
\end{align*}

where $U_i \in S_{2i}(t),\;i=1,2,3,4.$

In the same way, applying $T_{-U_3}\circ \rho_{4\pi/3}\circ T_{U_3}$ we obtain a new curvature line $S_3:I\to \mathbb R^4$ associated to $p=-\sqrt 3$ and given by:

\begin{align*}
  S_{31}(t)&=(-\frac{\sqrt{r^2-t^2 - t^3}}{\sqrt 2}, -\frac{\sqrt{r^2-t^2 + t^3}}{\sqrt 2},-\frac{t}{2},-\frac{\sqrt 3}{2}t),\\
  S_{32}(t)&=(\frac{\sqrt{r^2-t^2 - t^3}}{\sqrt 2}, -\frac{\sqrt{r^2-t^2 + t^3}}{\sqrt 2},-\frac{t}{2},-\frac{\sqrt 3}{2}t),\\
  S_{33}(t)&=(\frac{\sqrt{r^2-t^2 - t^3}}{\sqrt 2}, \frac{\sqrt{r^2-t^2 + t^3}}{\sqrt 2},-\frac{t}{2},-\frac{\sqrt 3}{2}t),\\
  S_{34}(t)&=(-\frac{\sqrt{r^2-t^2 - t^3}}{\sqrt 2}, \frac{\sqrt{r^2-t^2 + t^3}}{\sqrt 2},-\frac{t}{2},-\frac{\sqrt 3}{2}t),
\end{align*}

where $U_i \in S_{3i}(t),\;i=1,2,3,4.$ 
\end{proof}

\medskip

Proof of statement $iii)$. We begin by describing a $CW$ complex structure  which includes the $\nu$-lines of curvature 
and the $\nu$-umbilics as cells of $T_r(2)$.      

\begin{lem}\label{CW-structure}
The set $P_{++} \subset T_r(2)$ admits the following $CW$-complex structure.

\noindent i) Five 0-cells, one of them is an umbilic point.

\noindent ii) Five 1-cells, all of them are lines of curvature. Among them, two are separatrices containing the umbilic point.

\noindent iii) One 2-cell.  
\end{lem}

\begin{proof}
We begin by defining the 2-cell as the interior of $P_{++}$, see (\ref{pentagono}).
Then, the 1-cells are described as follows:
\begin{eqnarray*}
a_1& = & \{ (x,y,u,v) \in P_{++}; v=0\}, \ \ a_2=\{ (x,y,u,v) \in P_{++}; x=0, \ u>0\} \\
a_3& = &  \{ (x,y,u,v) \in P_{++}; y=0\}, \ \ a_4=\{ (x,y,u,v) \in P_{++}; x=0, \ v>0\} \\
&& \ \ \ \ \ \ \ a_5 =    \{ (x,y,u,v) \in P_{++}; v=-\sqrt{3}u\}.
\end{eqnarray*}
Observe that since the projection of $P_{++}$ onto the $uv$-plane is $P_0$, 
the projections of these 1-cells  are $e_1, ... , e_5$, respectively.
We point out that, according to the differential equation of $\nu$-lines of curvature of
Corollary $\ref{EcuaDiff}$, a straightforward computation implies that $a_i,\ i=1,...,5$ 
are $\nu$-lines of curvature of $T_r(2)$. Moreover, the curves $S_{13}$ defined in Proposition $\ref{circ1}$ and $S_{23}$ defined 
in Proposition $\ref{circ2}$
are  parameterizations of $a_1$ and $a_5$, respectively. Thus, these 1-cells are separatrices. Finally, the 0-cells are defined 
as the preimages of the vertices of the pentagonal region $P_0$ under the projection. Among them, the one corresponding to the origin 
in the $uv$-plane is the umbilic $(\frac{r}{\sqrt{2}}, \frac{r}{\sqrt{2}},0,0)$.
%\footnote{Ver SepXPedazos.pdf} 
\end{proof}

The $CW$-complex structure of $P_{++}$ in this lemma allows us to state the following: 

\begin{prop}
The double torus $T_r(2)$ admits a $CW$-complex structure such that the 0-skeleton includes the umbilic points and the 1-skeleton
the separatices.
\end{prop}
\begin{proof}
The cells of the $CW$-complex structure of $T_r(2)$ 
are the images of the cells of the pentagonal region $P_{++}$
under the action of the group $\mathcal G$. We describe them as follows: 
The 0-skeleton is constituted by sixteen 0-cells. Namely, four of the orbit of $v_0$, 
six for the orbit of $v_1$,
and three of the orbit $v_i, i=2,3$. Observe that the umbilic points belong to the orbit of $v_0$.
The 1-skeleton is constituted by thirty 1-cells defined by the orbits of the boundary of the pentagonal region. Namely,  
the orbit of $a_1$ consists of twelve 1-cells including $a_5$; all of them are separatrices. The orbit of $a_i$ consists of   
six 1-cells for each $i=2,3,4$.   
The 2-skeleton consists of
twelve 2-cells defined by the orbit of $P_{++}$. 
By a direct computation, it can be verified that this decomposition provides the Euler-characteristic of the double torus:
$$12 -30 +16= -2= \chi(T_r(2)).$$
\end{proof}

\begin{coro} 
The first homology group of the double torus $T_r(2)$, $H_1(T_r(2))$, with coefficients in $\mathbb Z_2$ 
has a basis whose representatives are cycles constituted by pathwise separatrices.
\end{coro}

\begin{proof}
We provide a basis $\{\alpha_i\}_{i=1}^4$ of $H_1(T_r(2))$ 
whose representatives are the cycles $\{c_i\}_{i=1}^4$, which are defined as follows.

\noindent First, we consider
\begin{eqnarray*}
c'_1= b_1 +b_2,\ c'_2= \Gamma_3(c'_1), 
\end{eqnarray*}
where $b_1= a_4 + \Gamma_3(a_2)$ and $b_2= \Gamma_1(b_1)$. 
Then, we define
\begin{eqnarray*}
c_1= d_1 +d_2+ \Gamma_1(d_1 +d_2),\ c_2= \Gamma_3(c_1), 
\end{eqnarray*}
where, $d_1$ is the separatrix whose projection is diagonal of slope $\sqrt{3},\ u,v \geq0$
and $d_2= \Gamma_3(d_1)$. 
Thus, $c_i,\ i=1,2.$ is a pathwise separatrix homologous to $c'_i,\ i=1,2.$ 
We define $d_3= a_3 + \Gamma_2(a_3)$ and $d_4= \Gamma_3(d_3)$.
%\begin{eqnarray*}
%c_3= d_3 +d_4,\ c_4= \Gamma_3(c_3). 
%\end{eqnarray*}
Thus, we have defined a basis $\{c_1', c_2', d_3, d_4  \}$  of  $H_1(T_r(2))$ whose cycles are constituted by 
$\nu$-lines of curvature and belong to the $1$-skeleton. 
Finally, by choosing the representatives 
$$c_3=a_1+a_5+ \Gamma _2(a_1+a_5)\ {\rm and}\ c_4=\Gamma_3(c_3),$$
since $c_i$ is homologous to $d_i$, $i=3,4$, we get the desired basis of $H_1(T_r(2))$ whose cycles $\{c_i \}_{i=1}^4$ are 
pathwise separatrices.
\end{proof}

Proof of statement $iv)$. In the proof of this statement, we will use some argument involving properties of the flow of a planar vector field. 
The Poincar\'e Bendixon theorem will play a relevant role. For the basic notions and the proof of this theorem see \cite{HirschSmale}.  

\begin{prop}\label{closedsolutions}
The lines of curvature on the complement of the separatices of $T_r(2)$ are immersions of $\mathbb S^1$. 
\end{prop}

\noindent {\bf Proof.} Let $\mathcal L_1$ and $\mathcal L_2$ be the foliations of maximal and minimal $\nu$-lines of curvature, respectively, in the closure of the 
region $P_0$  (see figure \ref{Pentagon}). We assume that $e_1$ and $e_5$ are separatrices in $\mathcal L_1$ containing the origin in the $uv$-plane. Thus, $e_2$ and $e_4$ belong to  
$\mathcal L_2$ and $e_3$ belongs in turn to $\mathcal L_1$
%\footnote{Observe that $e_2,e_3, e_4$ are lines of curvature since the satisfy the implicit equation...here that equation is useful}
. 
In the interior of $P_0$, the equation of  maximal (respectively minimal) $\nu$-lines of curvature is defined by a vector field $X_1$ without singularities.  
Moreover, we can assume that $X_1$ extends to $e_2$ pointing towards the region $P_0$ along $e_2$. Consequently, each solution with initial condition at $e_2$
never returns to this edge of $P_0$
%\footnote{The definition of this vector field can be given directly from the quadratic differential}
.

\begin{lem}
The $\omega$-limit of the solution $\gamma_p$ of $X_1$ with initial condition $ \gamma_p(0)=p \in e_2$ does not contain points in $e_1 \cup e_3\cup e_5$. 
Moreover, this solution intersects $e_4$.
\end{lem}
\noindent {\bf Proof of the lemma.} 
We show that there exists a tubular neighborhood of $e_1 \setminus \{(0,0)\}$ which does not intersect $\gamma_p(t),\ t>o.$
Consider the Poincar\'e application

\begin{eqnarray*} 
\pi: \Sigma_3^r(0) \rightarrow e_2
\end{eqnarray*}
where $\Sigma_3(0)$ is the minimal $\nu$-line of curvature in $P_0$ extending the separatrix in the complement of the closure of $P_0$, and where 
$\Sigma_3^r(0)$ is a neighborhood of the origin in $\Sigma_3(0)$ of ratio $r$. Since this application is continuous, we can choose $\delta>0$ such that
$\pi\left( \Sigma_3^{\delta}(0) \right) \subset e_2^{\epsilon=|p|/2}(p_0)$, where $p_0= e_1\cap e_2$ and $e_2^{\epsilon=|p|/2}(p_0)$ is the neighborhood 
of $p_0 \in e_2$ with ratio equal to a half of the distance from $p$ to $p_0$. Since $p \in e_2$ is not in $e_2^{\epsilon=|p|/2}(p_0)$, the 
uniqueness of the solutions of $X_1$ implies that $\gamma_p(0)$ never intersects the tubular neighborhood of $e_1$, defined by the solutions of $X_1$
with initial conditions in  $e_2^{\epsilon=|p|/2}(p_0)$. Analogous arguments allow us to extend such a tubular neighborhood to the whole union
$e_1 \cup e_3\cup e_5$. 

Now we prove the second statement. Suppose that $\gamma_p$ does not intersect $e_4$. Observe that
the $\omega$-limit of the solution $\gamma_p$ is closed
and contained in $P_0$ and thus bounded without singular points. Then the Poincar\'e-Bendixon theorem implies that $\gamma_p$ is a closed solution. 
This implies in turn that there is a singular point in the bounded region determined by this solution. So we get a contradiction.   
\qed

\begin{lem}\label{unicidad}
If $p,\ q \in e_2$ and $r \in e_4$ are such that $\gamma_p(t_0)= r= \gamma_q(t_1)$, where $t_i, i=0,1,$ is the first value 
of the parameter where the corresponding 
solution intersects $e_4$, 
then $p=q$ and $t_0=t_1$. 
\end{lem}
\noindent {\bf Proof of the lemma.}
\begin{eqnarray*}
\gamma_p:&[0,t_0]& \rightarrow \overline P_0,\ \gamma_p(0)=p, \ \gamma_p(t_0)=r \\
\gamma_q:&[0,t_1]& \rightarrow \overline P_0,\ \gamma_q(0)=q, \ \gamma_q(t_0)=r. 
\end{eqnarray*}
We define $i= \gamma_p[0,t_0] \cap \gamma_q[0,t_1]$. Thus, this set 
is a non-empty closed subset of each solution. Morever, the flow box theorem implies that $i$ is an open set. 
Then, $i = \gamma_p[0,t_0]$ and the proof follows from this equation. \qed

\noindent {\bf Proof of Proposition \ref{closedsolutions}.}

When the curve $\gamma_p(t)$ reaches the side $e_{4}$ at $r$ in the $uv$-plane, the curve $\overline \gamma(t)$ reaches the boundary of $H_{++}$ 
at the same point. Thus, 
%some point in the curve determined by the equation $x=0$.\\
in $T_r(2)$ the curve $\overline \gamma_p(t)=(x(t),y(t),u(t),v(t)),\ x(t)\geq 0$ crosses the boundary of $H_{++}$ and arrives to $H_{-+}$, now with $x(t) < 0$. 
Under the projection $\pi(x,y,u,v)=(u,v)$ it returns to the interior of the pentagon with negative values in the first entrance.
%$\overline{\gamma}(t)=(-x(t),y(t),u(t),v(t))$. 
%be the curve obtained from $\gamma(t)$ appliying a reflection in the hyperplane $x=0$. 
Note that $\pi \circ \overline \gamma = \gamma$. That is, the projection of $\pi\circ \gamma_p^-(t)$ is contained in  
$\gamma([0,t_0])$. Therefore, it returns to the pentagon following the solutions but in the opposite direction. 
So, Lemma \ref{unicidad} implies that this solution reaches $e_2$ at $p$ implying that $\overline \gamma_p$  is a closed solution. \qed

\begin{rem}
Observe that each line of curvature of the statement of the proposition is contained in the union 
$P_{++} \cup P_{-+}$, where these sets are the pentagonal regions defined in (\ref{Pentagon}), or in another pair of these regions.
\end{rem}

%\begin{rem}
%Let us provide an alternate argument to prove the second part of theorem $4.1$.
%Observe that since $M$ is $\mu$-umbilic, then, Theorem 3.4 \cite{R-S} implies that $M$ is totally semiumbilic and moreover there
%are a couple of foliations of orthogonal asymptotic lines. The fact that $\nu$ and $\mu$ are linearly independent at each point implies that
%the foliations of asymptotic lines coincides with that of $\nu$-lines of curvature.
%Therefore, theor%%%%%%%%%%%%%%%%%%%%%%%%%%%%%%%%%%%%%%%%%%%%%%%%%%%%%%%%%%%%%%%%%%%%%%%%%%%%%%%%%%%%%%%%%%%%%%%%%%%%%%%%%%%%%%%%%%%%%%%%%%%%%%%%%%%%%%%%%%%%%%%%%%%%%%%%%%%%%%%%%
\subsection{The lines of curvature of the double torus in $\mathbb R^3$}

%%%%%%%%%%%%%%%%%%%%%%%%%%%%%%%%%%%%%%%%%%%%%%%%%%%%%%%%%%%%%%%%%%%%%%%%%%%%%%%%%%%%%%%%%%%%%%%%%%%%%%%%%%%%%%%%%%%%%%%%%%%%%%%%%%%%%%%%%%%%%%%%%%%%%%%%%%%%%%%%%

We consider a surface $M$, defined as the transversal intersection of the sphere of radius $r$ centered at the origin, $\mathbb S^3_r \subset \mathbb R^4$, 
with the inverse image of a value of a differentiable function $F:\mathbb R^4 \rightarrow \mathbb R$
with an isolated singularity at the origin. Let $\nu$ be a unit vector field parallel to the gradient of $F$. 
The lines of curvature transform in a proper way under the stereographic projection. We state the property in a more general setting.

\begin{prop}
Consider a surface $M \subset \mathbb S^3_r$. 
The stereographic projection $\sigma:\mathbb S^3_r\setminus{\{p\}} \subset \mathbb R^4 \rightarrow \mathbb R^3$, where $p \in \mathbb S^3_r \setminus M$, provides an immersion 
of $M$ into $\mathbb R^3$, which transforms the $\nu$-lines of curvature of $M$ into the lines of curvature of $\sigma(M)$, for any field $\nu$ normal to $M$ not parallel to the radial vector field. That is, a curve 
$l \subset M$ is a  $\nu$-line of curvature if and only if $\sigma(l)$ is a line of curvature of $\sigma(M)$. 
\end{prop}

%\footnote{Define asymptotic lines and point out the coincidence of the curvature lines and asymptotic lines in this setting.}

\noindent{\bf Proof.}
We begin by recalling that any spherical surface $M \subset \mathbb R^4$ is $R$-umbilic, where $R$ is the {\it radial vector field}, defined as the 
unitary vector field parallel to the gradient of the function $G$. 
Moreover, any normal field independent to the radial vector field at each point of $M$ 
defines the same principal configuration (see Lemma $2.1$ in \cite{R-S}). Consider this normal field $\nu$. Its orthogonal projection to 
$T_pS^3_r$ defines a non-null field $\hat \nu$ normal to $M$ at each point. 
%Let us consider first the sphere $\mathbb S^3_r \subset \mathbb R^4$ as the inverse image of zero of the function:
%$G_r:\mathbb R^4 \rightarrow \mathbb R,\ \ G_r(x,y,u,v)=x^2+ y^2 + u^2 + v^2 - r^2 $.
%Its gradient is parallel to the radial vector field $R$, namely, it is the position vector. A direct computation implies 
%that $M$ is $R$-umbilic. Because $\nu$ and $R$ are linearly independent at each point of $M$, implies that
%the $\nu$-principal configuration coincides with that defined by any of the two unit vector fields normal to $M$ but tangent to $\mathbb S^3_r$
%at each point of $M$. Denote by $\mu$ this unit vector field    
The stereographic projection %$\sigma:\mathbb S^3 \rightarrow \mathbb R^3$ 
as a conformal map transforms
the vector field $\hat \nu$ into one of the normal fields of $\sigma(M)\subset \mathbb R^3$, and the $\hat \nu$-lines of curvature on $M$ into
the lines of curvature $\sigma(M)$. Since the $\nu$-lines of curvature of $M$ coincide with the $\hat \nu$-lines of curvature, the proof is complete.
$\hfill\Box$
%em XXX proven in \cite{RF} guarantees that the asymptotic lines are transformed by
%the inverse of the stereoghapic projection into the curvature lines of $\sigma(M) \subset \mathbb R^3$.
%\end{rem}

We denote the embedding $\sigma(T_r(2)), r= 1$ of the double torus into  $\mathbb R^3$ by $\mathbb T_2$.

\begin{coro}\label{torusinR3}
The principal configuration of the embedded double torus $\mathbb T_2 \subset \mathbb R^3$ has
the same structure of that of $T_r(2)$ provided by Theorem \ref{Maintheorem}.
\end{coro}

We conclude by observing that the approach applied in this article to study
the lines of curvature of the double torus can be extended to describe the
lines of curvature of model embeddings of compact oriented surfaces of
genus higher than two. Specifically, we have used the differential equation
in Corollary 2.5 to describe the $\nu$-lines of curvature of the model defined
in section 3 as the link $L_r = Re \bar F \cap {\mathbb S}^3, r > 0$, of genus $(p-1)(q-1),
p, q \geq 3$, where the vector field $\nu$ is the gradient of 
$Re \bar F = z_1^p + z_2^q + \bar z_1^p + \bar z_2^q$. 
The symmetries of this model allowed us to determine the
type of umbilic points. Moreover, the topological decomposition of the link
induced a CW-structure which includes umbilic points and separatrices as
some of its cells. Furthermore, arguments of dynamical systems, similar to
those used in this article, could be applied to prove that the lines of
curvature are closed cycles. Finally, as in the case of the double torus,
the stereographic projection defines the embedding of $L_r, r > 0$ into $\mathbb R^3$
which transforms the foliations of $\nu$-lines of curvature into the foliation
of lines of curvature of the image of the link in $\mathbb R^3$. However, certain
problems for the determination of the number of umbilic points appeared in
the case of an arbitrary genus, which need to be solved to complete the
description.

\medskip
\textbf{Acknowledgement:} 
The work of the second and third authors was supported by CONACYT grant 283017, the
third author was also supported by PAPIIT-DGAPA-UNAM grant IN118217.

\end{document}